\documentclass[letterpaper, 10 pt, conference]{ieeeconf}  

\IEEEoverridecommandlockouts                              
\usepackage{graphicx} 
\usepackage{epstopdf}
\usepackage{amsmath}
\usepackage{algcompatible}
\usepackage{algpseudocode}
\usepackage{algorithm,algorithmicx,listings}
\usepackage{amssymb}
\usepackage{enumerate}
\usepackage[usenames, dvipsnames]{color}
\usepackage{algorithm}
\usepackage{algpseudocode}
\usepackage{mathtools}
\usepackage{breqn}
\usepackage{bbm}
\usepackage{cite}
\usepackage{pdfpages}
\usepackage{cuted}
\usepackage{braket}
\usepackage{xifthen}
\usepackage{needspace}
\usepackage{array}
\usepackage{}
\newtheorem{remark}{\bf Remark}

\newtheorem{lemma}{\bf Lemma}
\newtheorem{proposition}{\bf Proposition}

\usepackage{array}
\newcolumntype{"}{@{\hskip\tabcolsep\vrule width 2pt\hskip\tabcolsep}}

\newcounter{protocol}
\newenvironment{protocol}[1][]%
  {
     \needspace{2\baselineskip}
    \noindent \rule{\linewidth}{1pt} \endgraf
    \refstepcounter{protocol}
    \centering \textsc{Protocol}~\theprotocol%
    \ifthenelse{\isempty{#1}}{}{:\ #1}
  }{
  \vspace{-2pt}
  \noindent \rule{\linewidth}{1pt}\vspace{2pt}
  }

\newcommand{\mr}{\mathrm}
\newcommand{\mb}{\mathbf}
\newcommand{\mbb}{\mathbb}

\newcommand{\mbR}{\mathbb R}

\newcommand\numberthis{\addtocounter{equation}{1}\tag{\theequation}}

\title{\LARGE \bf
Encrypted Distributed Lasso for Sparse Data Predictive Control
}

\author{Andreea B. Alexandru \and Anastasios Tsiamis \and George J. Pappas%
\thanks{The authors are with the Department of Electrical and Systems Engineering, University of Pennsylvania, Philadelphia, PA 19104.
        {\tt\small \{aandreea,atsiamis,pappasg\}@seas.upenn.edu}}%
}

\begin{document}

\maketitle
\thispagestyle{plain}
\pagestyle{plain}

\begin{abstract}
The least squares problem with $\ell_1$-regularized regressors, called \emph{Lasso}, is a widely used approach in optimization problems where sparsity of the regressors is desired. This formulation is fundamental for many applications in signal processing, machine learning and control. As a motivating problem, we investigate a sparse data predictive control problem, run at a cloud service to control a system with unknown model, using $\ell_1$-regularization to limit the behavior complexity. The input-output data collected for the system is privacy-sensitive, hence, we design a privacy-preserving solution using homomorphically encrypted data. The main challenges are the non-smoothness of the $\ell_1$-norm, which is difficult to evaluate on encrypted data, as well as the iterative nature of the Lasso problem. We use a distributed ADMM formulation that enables us to exchange substantial local computation for little communication between multiple servers. We first give an encrypted multi-party protocol for solving the distributed Lasso problem, by approximating the non-smooth part with a Chebyshev polynomial, evaluating it on encrypted data, and using a more cost effective distributed bootstrapping operation. For the example of data predictive control, we prefer a non-homogeneous splitting of the data for better convergence. We give an encrypted multi-party protocol for this non-homogeneous splitting of the Lasso problem to a non-homogeneous set of servers: one powerful server and a few less powerful devices, added for security reasons. Finally, we provide numerical results for our proposed solutions.
\end{abstract}

\IEEEpeerreviewmaketitle

\section{Introduction}
\label{sec:introduction}
Sparsity and compressed sensing have been widely used in signal processing, machine learning and control applications, especially in the big-data regime and noisy environments~\cite{candes2006stable,hastie2015statistical}. In high-dimensional problems, it is likely that only a subset of features affects the observations. Hence, pursuing sparse representations reduces the model complexity, prevents overfitting and helps with overall interpretation. Since finding the solution with the minimal number of non-zero coefficients~is NP-hard, $\ell_1$-regularization has been proposed as a convex method that promotes sparsity. Perhaps one of the most used algorithms for sparse recovery has been the celebrated Lasso algorithm (least absolute shrinkage and selection operator), which accounts for both sparsity and potentially noisy data, using $\ell_1$-regularization. 

For instance, Lasso has been used in signal reconstruction for medical imaging, wireless communication and tracking; seismology applications; portfolio optimization; text analysis~\cite{zhang2015survey,hastie2015statistical}.
Many of these applications are large-scale or involve data coming from multiple data sources.
With the recent widespread availability and development of cloud services, it seems an attractive and cost effective solution to outsource the computations to the cloud, when the data owner or querier lacks the computational resources and/or expertise to locally perform them. 
Given the privacy-sensitive nature of medical data, financial data, location data, energy measurements etc., on which such problems are computed, and how they can be used to profile users or mount attacks on critical infrastructure, the computations should not be performed in the clear at the cloud service.

\subsection{Contributions}\label{subsec:contributions}
In order to deal with the privacy issues, we draw on cryptographic approaches, specifically, on homomorphic encryption, which enables polynomial computations by the cloud over encrypted data of the client. 
However, encrypted Lasso brings new challenges: evaluating non-smooth functions on encrypted data, as well as continuing computations over multiple iterations and time steps, which generally requires refreshing the ciphertexts.

A conventional observation is that distributing a large optimization problem to multiple servers improves the execution time by parallelizing smaller subproblems. 
Apart from this, we note that \emph{distributing the computation allows a streamlined execution of encrypted iterations}. In particular, using multiple servers allows us to perform a refresh operation at a substantially reduced cost compared to performing it only at one server. This cheaper refresh operation enables us to continue the encrypted computations over multiple iterations, as well as to use a high degree polynomial to approximate the gradient of the $\ell_1$-norm. 
Specifically, we propose:
\begin{itemize}
    \item an efficient distributed encrypted solution to Lasso problems using ADMM, offering computational privacy of all the data, including intermediate results;
    \item an optimized implementation of the above protocol using an efficient Chebyshev series evaluation for polynomial approximations and reducing the number of ciphertext levels and operations.
\end{itemize}

We apply our cloud-based sparsity framework to the problem of data-based predictive control.
Our goal is to control an unknown system using only the privacy-sensitive input-output data that are potentially noisy. Data-driven control is a blooming research area and methods based on the behavioral framework have received significant renewed interest \cite{de2019persistency,Coulson2019data,Berberich2020data,coulson2020distributionally,VanWaarde2020data,berberich2020robust,dorfler2021bridging} since their original proposal~\cite{Willems2005note,markovsky2008data}. The idea of such methods is that the state representation can be replaced by a data-based representation which only uses the trajectories of the system, bypassing the need for system identification. In the case of noisy data, inspired by~\cite{Coulson2019data,dorfler2021bridging}, we reformulate the data-based predictive control as a lasso problem.

For this use-case, we propose:
\begin{itemize}
    \item a distributed encrypted solution for $\ell_1$-regularized data predictive control, using an optimized implementation.
\end{itemize}
For better convergence, we customize this solution to split the problem heterogeneously between a powerful server and a few less powerful machines.

\subsection{Comparison to related work}\label{subsec:related}

The usage of ADMM for private distributed optimization is not novel, see e.g.,~\cite{zhang2018admm,ye2020privacy}, given its convenient formulation and splitting of the objective function and variables. (For other distributed gradient based methods, see e.g.,~\cite{Lu2018privacy}.) However, our usage of distributed ADMM substantially differs from previous works in the following: i) we start with centralized rather than already distributed data, so we split the centralized problem in a way that fits our privacy and low-power requirements; ii) we assume heterogeneous servers and we split the computations differently depending on who performs them; iii) the data at each server is not in the clear, which complicates the computations; iv) the servers do not learn any of the data, including intermediate iterates and results; this requires more complex computations to privately perform nonlinear operations; v) the $\ell_1$-regularization term is non-smooth and has nonlinear gradient, leading to updates of the global primal variable that are incompatible with the mentioned ADMM works.

Compared to~\cite{Alexandru2020cloud}, we consider the approximation of a non-smooth nonlinear objective function and use multiple servers to streamline the complex iterative computation, rather than a two-party computation, and a more powerful homomorphic encryption scheme to partially replace the blinded communication necessary at every iteration in~\cite{Alexandru2020cloud}.

The authors in~\cite{Zheng2019helen} propose a distributed ADMM for a Lasso problem, using an threshold additively homomorphic encryption and SPDZ~\cite{damgaard2012multiparty} for computing the nonlinear operations. 
In contrast to their work, the data is not distributed in the clear to the computing servers, meaning we have less flexibility with respect to the local computations. However, we can choose to split the data in the most convenient way to have a distributed convergence speed similar to the centralized convergence, which~\cite{Zheng2019helen} does not discuss. Another difference is that the tools they use require them to communicate for every nonlinear computation, such as multiplications and comparison operations, which require a number of communication rounds dependent on the number of bits in the messages. In our case, the servers only send two messages per iteration and the method we employ also allows us to batch vectors and perform operations in parallel for all elements of a vector. 

In~\cite{froelicher2021scalable,sav2020poseidon}, the authors propose distributed/federated training and evaluation with multi-party fully homomorphic encryption for linear, logistic and neural network models, using stochastic gradient descent. While we inspired our solution from the multi-party fully homomorphic encryption tool, their setup is different than ours: the data is either distributed in cleartext locally at the servers or other data providers perform the preprocessing; and the computations are different, leading to different optimizations: e.g., \cite{froelicher2021scalable} uses a combination of distributed and centralized bootstrapping operations, \cite{sav2020poseidon} has only one recurring variable (the model) to bootstrap. 

While differential privacy for the solution of the optimization problem~\cite{Han17, zhang2018improving} is out of the scope of this paper, it is an avenue for future research in privately determining parameters of the optimization problem. 

Encrypted control, surveyed in~\cite{Darup2020encrypted}, has been recently gaining momentum, due to the strong privacy guarantees it offers even when the controller is located on an untrusted platform. 
In the case of controlling a system with known model and linear controller parameters, the line of work~\cite{Cheon2018need,kim2019dynamic} has shown how to perform the computations at subsequent time steps without the need of bootstrapping or ciphertext reset. In contrast, we deal with both unknown model matrices and nonlinearities, which prevent the application of their methods. 

Our previous work~\cite{Alexandru2020towards,Alexandru2020data} provides confidentiality for~a different formulation of a data predictive control problem. The Lasso formulation in our current work does~not~have~a closed-form solution as before, which complicates the computations on encrypted data. We also use a different architecture at the cloud, in order to completely remove the client involvement during the computation of the control input.

An interesting remark is related to~\cite{Kim16encrypting}, where the authors propose to use multiple controllers in parallel that perform asynchronous local bootstrapping to ensure that at least one has a control input ready at each time step. In our case, we prefer multiple servers to perform a distributed bootstrapping in order to reduce the time it takes to refresh the ciphertexts.

\section{Problem formulation}
\label{sec:formulation}
For a covariate matrix $\mb A\in \mbb R^{m\times n}$, a vector of outcomes $\mb b\in\mbb R^m$, the variable $\mb x\in\mbb R^n$ and a penalty parameter~$\lambda>0$, the Lasso problem in its Lagrangian (unconstrained) form is given by:
\begin{equation}\label{eq:Lasso}
    \min_{\mb x}~ \frac{1}{2}  ||\mb A\mb x- \mb b||_2^2 + \lambda || \mb x ||_1.
\end{equation}

For dependent covariates, there is no closed-form solution to~\eqref{eq:Lasso} and many iterative optimization algorithms have been proposed in the literature~\cite[Ch.~5]{hastie2015statistical}. 
For example, Lasso problems can be solved using proximal methods or augmented Lagrangian methods, such as the Alternating Direction Method of Multipliers (ADMM)~\cite{boyd2011distributed},~\cite[Ch.~5]{hastie2015statistical}. 
Splitting the objective function in the ADMM way, we get:
\begin{align}\label{eq:Lasso-split}
\begin{split}
\min_{\mb x,\mb z}~& \frac{1}{2}  ||\mb A\mb x- \mb b||_2^2 + \lambda || \mb z ||_1  \\
\text{s.t.}~&~ \mb x - \mb z = \mb 0.
\end{split}
\end{align}

Let $S_\alpha(\mb x) = (\mb x - \alpha \mb 1)_{+} - (-\mb x - \alpha \mb 1)_{+}$ denote the soft thresholding operator. The ADMM algorithm for~\eqref{eq:Lasso-split} is:
\begin{align*}
    \mb x^{k+1} &= \arg\,\min\limits_{\mb x} \left( \frac{1}{2} || \mb A \mb x - \mb b ||_2^2 + \frac{\rho}{2} ||\mb x - \mb z^k + \mb w^k ||_2^2 \right) \\
    &= \left(\mb A^\intercal \mb A + \rho \mb I\right)^{-1} \left( \mb A^\intercal \mb b + \rho (\mb z^k - \mb w^k) \right) \\
    \mb z^{k+1} &= \arg\,\min\limits_{\mb z} \left( \lambda || \mb z ||_1 + \frac{\rho}{2} ||\mb x^{k+1} - \mb z + \mb w^k ||_2^2 \right)\numberthis\label{eq:admm}\\
    &=S_{\lambda/\rho} (\mb x^{k+1} + \mb w^k) \\
    \mb w^{k+1} &= \mb w^k + \mb x^{k+1} - \mb z^{k+1}.
\end{align*}

While for general optimization problems, the ADMM might converge slowly, for the Lasso problem it is known to have a fast convergence of a few (tens of) iterations for a large range of parameter $\rho>0$~\cite{boyd2011distributed}. ADMM is a fast choice in the cases where a very high precision of the optimal result is not required, which is the case in noisy control problems, like the one we investigate in Section~\ref{sec:reformulation}.

\textbf{Goals and privacy requirements. }
A client outsources problem~\eqref{eq:Lasso} to a cloud service that has to compute the optimal solution based on the data from the client. 

We consider the cloud service to be \textit{semi-honest}, meaning it does not deviate from the client's specifications, but can process the data it receives to extract private information for its own profit. 
The cloud service can be a conglomerate~of~$K$ servers (see Figure~\ref{fig:setup}), possibly belonging to different~organizations, offering the guarantee that not all $K$ servers collude.

Under this adversarial model, we require \textit{client data confidentiality}, i.e., an adversary corrupting at most $K-1$ of the servers should not be able to infer anything about the client's inputs and outputs, which consist of the values of the matrix $\mb A$ and the vector $\mb b$, any intermediate values, and solution $\mb x^\ast$. The penalty $\lambda$ and parameter $\rho$ can be chosen by the cloud service or chosen by the client, but are public. 

\begin{figure}[htb]
	\centering
	\includegraphics[width=0.7\columnwidth]{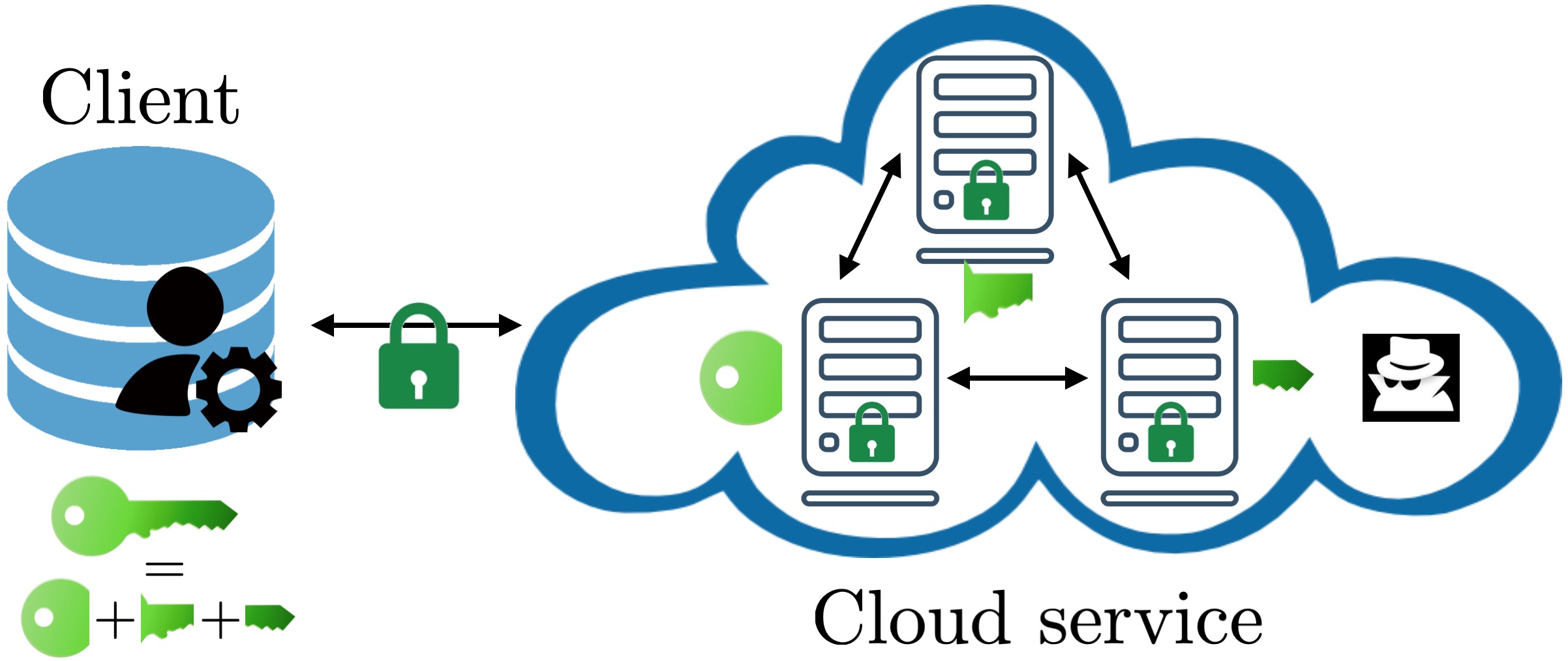}
	\caption{Schematic diagram of the problem, with a client outsourcing its encrypted data to a cloud service, that is authorized to compute on the data, but not to decrypt it. }
	\label{fig:setup}
\end{figure}

\section{Preliminaries}
\label{sec:preliminaries}

\subsection{Homomorphic Encryption (HE)}
\label{subsec:FHE}
A HE scheme is called \textit{partially homomorphic} if it supports the encrypted evaluation of either a linear polynomial or a monomial, \textit{somewhat or leveled homomorphic} if it supports the encrypted evaluation of a polynomial with a finite degree 
and \textit{fully homomorphic} if it supports the encrypted evaluation of arbitrary polynomials.

A HE scheme has a \textit{computational security parameter} $\kappa$ if all known attacks against the scheme, including breaking the encryption or distinguishing between the encryptions of different messages, take $2^\kappa$ bit operations. 
In practice, at least 128 bits of security are preferred~\cite{albrecht2019homomorphic}. 

The newer generations of HE schemes are based on lattices and rely on the Ring Learning with Errors~\cite{Lyubashevsky2010ideal,Albrecht15} hardness problem. Each operation on ciphertexts introduces some noise, with additions being cheap in terms of noise amount, but multiplications being expensive. If the noise amount exceeds a threshold, decryption yields an incorrect result, so parameters of HE schemes need to be designed accordingly. 
The \textit{level} in leveled HE schemes concerns the supported number of sequential multiplications. We say that a fresh ciphertext is on level $L$ and a multiplication consumes a level. A ciphertext on level 0 does not accept any more multiplications. 
Leveled HE schemes can be turned into fully HE schemes if a \textit{bootstrapping} operation is enabled, which can \textit{refresh} the ciphertext after the levels were consumed, such that further operations are allowed while still guaranteeing correct decryption.

If done locally at a server with no access to the private key, bootstrapping is a very expensive operation, consuming around 10 levels and introducing more noise (from encrypted approximation of non-polynomial functions); see~\cite{Cheon2018bootstrapping,Chen2019improved} for details. Apart from the computationally intensive bootstrapping procedure, all the prior and posterior operations are impacted, since ciphertexts are required to have an extra 10 levels, leading to very large scheme parameters and ciphertexts size, making centralized bootstrapping undesirable.

Instead of performing bootstrapping locally, a server can ask the client to refresh a ciphertext on level 0. However, this implies more computation, communication and availability from the client, which is often prohibitive. A preferable solution is to use two or more servers for the computation and the refreshing. 
However, corrupting only two servers can~be~attainable by an adversary. Increasing the number of servers decreases the probability of an adversary corrupting them all. \emph{Distributed bootstrapping trades substantial computation power to one round of communication and does not introduce as much noise as the centralized bootstrapping}, as described below.

In multi-party HE schemes, the private key is additively secret shared between a number of servers, meaning that no proper subset of the servers can decrypt. An important assumption is that an adversary cannot corrupt all servers at once, hence the private key is never recovered. In~\cite{mouchet2020multiparty}, a multi-party HE scheme is described, where servers can carry out the homomorphic computations locally and only need to interact for decryption and bootstrapping. In our scenario, the decryption will be performed at the client so we are only interested in distributed bootstrapping. Intuitively, distributed bootstrapping requires each server to use its local secret share of the private key to perform a partial decryption, mask this result and send it to the other servers. Summing up all the partial decryptions results in a refreshed ciphertext with the desired number of levels that can be correctly decrypted to the original message. The masking needs to provide statistical privacy of the message, so we require the mask to be $>80$ bits larger than the size of messages and that no overflow occurs (the statistical security parameter is generally smaller than the computational security parameter). This means that distributed bootstrapping consumes around 3 levels; see more details in~\cite{mouchet2020multiparty, froelicher2021scalable}. 

In this paper, we work with the multi-party version~\cite{froelicher2021scalable}~of the leveled HE scheme CKKS~\cite{Cheon2017homomorphic}. 
Each plaintext/ciphertext is a polynomial in the ring of integers of a cyclotomic field. This enables the encoding of multiple scalars in a plaintext/ciphertext and performing single instruction multiple data ($\mr{SIMD}$) operations, which can bring major computation and storage improvements. 
Abstracting the details away, the $\mr{SIMD}$ operations that can be supported are addition, element-wise multiplication by a plaintext or ciphertext and data slot rotations in ciphertexts. We will use $+$ and $\odot$ for $\mr{SIMD}$ addition and multiplication and $\rho(\mb x,i)$ to denote the row vector $\mb x$ rotated to the left by $i$ positions ($i<0$ means rotation to the right). 

We will denote by $\mr{E_{v0}}(\mb x)$ the encryption of the vector~$\mb x$ followed by trailing zeros and by $\mr{E_{v\ast}}(\mb x)$ the encryption of the vector $\mb x$ followed by junk elements (elements whose value we do not care about).

\subsection{Polynomial approximation and Chebyshev series}
\label{subsec:poly_eval}

As described above, homomorphic encryption can evaluate polynomials on encrypted values. However, other operations such as trigonometric functions, divisions or comparisons are not supported. Therefore, we prefer to evaluate a polynomial approximation of the non-polynomial functions. We choose to work with the Chebyshev polynomial series rather than the more common Taylor power series due to better precision and smaller approximation error. Specifically, the Chebyshev series polynomial interpolation is a near-minimax approximation of a continuous function on the interval $[-1,1]$~\cite{mason2002chebyshev}.

However, polynomial approximation is not a panacea: for non-smooth functions, it gives a reasonable error only on relatively small intervals and using high degree polynomials. We choose to use this method, rather than other encrypted computation tools that can exactly evaluate non-smooth functions at the cost of more communication, knowing that we are dealing with noisy systems, where the small approximation errors are absorbed by noise.

\section{Encrypted distributed Lasso}
\label{sec:solution}
The setting we consider is the following: the client encrypts its data $\mb A, \mb b$ and secret shares its private key to a number of servers. The servers are responsible to compute and return the solution of problem~\eqref{eq:Lasso-split} to the the client.

\subsection{Challenge: Evaluating non-polynomial functions}
In the steps~\eqref{eq:admm} of the ADMM algorithm for problem~\eqref{eq:Lasso-split}, the soft thresholding function is non-polynomial, yet we need to evaluate it on encrypted data when computing $\mb z^{k+1}$. We deal with this challenge by approximating the soft~thresholding function using a polynomial on a fixed interval via the Chebyshev series (we hardcode the coefficients for this function). If the interval is not $[-1,1]$, we first apply a linear transformation to bring the inputs to this interval.

In the context of encrypted evaluation, a high polynomial degree increases the number of levels necessary for the computations, as well as the number of homomorphic multiplications between ciphertexts, which are expensive operations (compared to plaintext-ciphertext multiplications or additions). While we cannot consume fewer than $\left\lceil \log n \right\rceil$ levels to evaluate a polynomial of degree $n$, we can reduce the $O(n)$ homomorphic multiplications from the naive evaluation. Specifically, we implement the Paterson-Stockmayer algorithm~\cite{paterson1973number}, which reduces the number of homomorphic multiplications to $\left\lceil \sqrt{2n} + \log n \right\rceil + \mathcal O(1)$ by recursively evaluating polynomials of smaller degree. We modify the Paterson-Stockmayer algorithm that works with power series to work with Chebyshev series. The benefit of this algorithm compared to the naive evaluation is visible after degree 5 and grows with the degree. As an example, to evaluate a non-monic polynomial of degree 25, we require 5 levels and 11 homomorphic multiplications between ciphertexts. 

\begin{remark}
The ``stability" of the ADMM iterations allows the value $\mb x^{k+1} + \mb w^k$ to stay within a fixed interval, given in Lemma~\ref{lemma:interval}. In practice, we choose this interval from prior simulation.
\end{remark}

\begin{lemma}\label{lemma:interval}
Define $\mb M:= \mb A^\intercal \mb A + \rho \mb I$, $\mb n: = \mb M^{-1}\mb A^\intercal \mb b$ and $c := \sqrt{n}\lambda/\rho ||2\rho \mb M^{-1} - \mb I||_2 + ||\mb n||_2$. Since $\sigma:=||\rho \mb M^{-1}||_2$ is in $(0,1]$, we have the following bounds for the quantity $||\mb x^{k+1} + \mb w^k||_\infty$ in~\eqref{eq:admm}, for all $k=1,\ldots,K_{\mr{iter}}$:
\begin{equation}\label{eq:bound}
\begin{aligned}
  ||\mb x^{k+1} + \mb w^k||_\infty &\leq \sigma^k  ||\mb n||_2 + \frac{1-\sigma^k}{1-\sigma}c, &&\text{if } \sigma<1\\
  ||\mb x^{k+1} + \mb w^k||_\infty &\leq ||\mb n||_2 + kc , &&\text{if } \sigma=1.
\end{aligned}
\end{equation}
\end{lemma}
\vspace{7pt}

\begin{proof}
Let $\mb y^k:= \mb x^{k} + \mb w^{k-1}$ Manipulating~\eqref{eq:admm}, we get:
\begin{align*}
    \mb y^{k+1} &= (\mb I - \rho \mb M^{-1})\mb y^{k} + (2\rho \mb M^{-1} - \mb I)\mb z^k + \mb n\\
    &= \rho \mb M^{-1} \mb y^k + (2\rho \mb M^{-1} - \mb I)(\mb z^k - \mb y^k) + \mb n.
\end{align*}
The expression of the thresholding operator gives the following bound: $-\lambda/\rho \leq \mb z_i^k - \mb y_i^k \leq \lambda/\rho$, for $i=1,\ldots,n$. Then, using the triangle inequality and submultiplicative property:
{\medmuskip=1mu
\begin{align*}
    &||\mb y^{k+1}||_2 \leq ||\rho\mb M^{-1}\mb y^{k}||_2 + ||(2\rho \mb M^{-1} - \mb I) (\mb z^k - \mb y^k)||_2 + ||\mb n||_2\\
    &\leq ||\rho \mb M^{-1}||_2 ||\mb y^k||_2 + \sqrt{n}\lambda/\rho ||2\rho \mb M^{-1} - \mb I||_2 + ||\mb n||_2.
\end{align*}
}\ignorespaces
We select $\mb z^0 = \mb w^0 = \mb 0$, compress the geometric~progression and use $||\mb y^{k+1}||_\infty \leq ||\mb y^{k+1}||_2$, getting the bounds in~\eqref{eq:bound}.
\end{proof}

\subsection{Challenge: Evaluating iterations}
Depending on the precision we choose, the polynomial approximation can have a high degree, implying the need of bootstrapping in order to continue operations in the subsequent iterations. We resolve this challenge by making use of multiple servers in order to realize a cheaper bootstrapping compared to a centralized bootstrapping and a less burdensome solution than requesting action from the client. 

To this end, we turn to the distributed version of ADMM \cite{boyd2011distributed}, such that we use the servers both to ease the computation of the optimal solution and to ensure privacy through encrypted computations. We split the matrix $\mb A$ and vector $\mb b$ into $K$ parts, each to be held by a server, and rewrite~\eqref{eq:Lasso-split} as: 
\begin{align}\label{eq:Lasso-split-dist}
\begin{split}
\min_{\mb x_1,\ldots, \mb x_K, \mb z}~& \frac{1}{2}  \sum_{i=1}^K||\mb A_i\mb x_i-\mb b_i||_2^2 + \lambda || \mb z ||_1  \\
\text{s.t.}~&~ \mb x_i - \mb z = \mb 0, \quad i=1,2,\ldots,K.
\end{split}
\end{align}
The ADMM algorithm for problem~\eqref{eq:Lasso-split-dist} is, for $i = 1,\ldots,K$:
\begin{align}\label{eq:ddc-admm-dist}
\begin{split}
    \mb x_i^{k+1} &= \left(\mb A_i^\intercal \mb A_i + \rho \mb I\right)^{-1} \left( \mb A_i^\intercal \mb b_i + \rho (\mb z^k - \mb w_i^k) \right) \\
    \mb z^{k+1} &= \frac{1}{K}S_{\lambda/\rho} \left(\sum_{i=1}^K\mb x_i^{k+1} + \sum_{i=1}^K\mb w_i^k\right) \\
    \mb w_i^{k+1} &= \mb w_i^k + \mb x_i^{k+1} - \mb z^{k+1}.
\end{split}
\end{align}

Each server is given ciphertexts corresponding to $\mb A_i, \mb b_i$. We assume that there is a preprocessing step where servers can precompute convenient ciphertexts that will be used often in the online iterations, such as $1/\rho\mb A_i^\intercal \mb b_i$ and $\rho\left(\mb A_i^\intercal \mb A_i + \rho \mb I\right)^{-1}$. As in the unencrypted case, the servers can use the matrix inversion lemma to compute an inversion of a smaller matrix, which saves in offline computation. Online, each server locally computes the encryptions of $\mb x_i$ and $\mb w_i$, then communicates to the other servers the local sum $\mb x_i^{k+1} + \mb w_i^k$, such that all servers are then able to compute $\mb z^{k+1}$. So far, the only communication necessary is the same as in the unencrypted ADMM.

\subsection{Challenge: Realizing the fewest bootstrapping operations}
Distributed bootstrapping requires all parties to start by holding the same ciphertext and all parties to obtain that refreshed ciphertext. Bootstrapping the ciphertext encrypting $\mb z^{k+1}$ seems attractive, because it is global and its evaluation involves the most sequential multiplications. However, this is not enough: $\mb w_i^{k+1}$ loses levels through $\mb x_i^{k+1}$, which is the result of a multiplication; so we would need to bootstrap also before computing $\mb z^{k+1}$, not just after.

Instead, we do the following trick. Each server already~has to compute and send a ciphertext encrypting~$\mb x_i^{k+1} + \mb w_i^k$ to the other servers in order to compute the global iterate $\mb z^{k+1}$. This means that each server can then construct a packed ciphertext $c^{k+1}$ encrypting $\big[(\mb x_1^{k+1} + \mb w_1^k)^\intercal (\mb x_2^{k+1} + \mb w_2^k)^\intercal \, \ldots \, (\mb x_K^{k+1} + \mb w_K^k)^\intercal \big]$ and distributedly bootstrap it. Afterwards, each server can extract the refreshed ciphertext containing its local value $\mb x_i^{k+1} + \mb w_i^k$, as well as a refreshed ciphertext containing $\sum_{i=1}^K \mb x_i^{k+1} + \mb w_i^k$ by repeatedly rotating and summing the refreshed ciphertext $c^{k+1}$ (this takes $O(K)$ operations). From this value, each server can locally compute its encrypted iterates $\mb w_i^{k+1}$ and $\mb x_i^{k+1}$, while doing only one bootstrapping operation per iteration rather than two. 

Apart from $\mb x_i^{k+1} + \mb w_i^k$, the servers send one more message to complete the bootstrapping operation, so there are two rounds of communication per iteration, one at the smallest admissible level (dictated by bootstrapping) and the other at the full number of levels required, computed below. 

Assume that offline quantities are refreshed. Define~$l_B$ to be the number of levels for a statistically secure distributed bootstrapping and $l_P$ to be the number of levels necessary for the evaluation of $S_{\lambda/\rho}(\cdot)$ at a desired precision: this is the degree of the approximation polynomial plus one, coming from the linear transformation to the interval $[-1,1]$ (we merge the multiplication by $1/K$ in the Chebyshev coefficients). Hence, the fresh ciphertexts need to have $L = l_B + l_P + 1$ levels, if we bootstrap once per iteration. Because $l_P$ is usually higher than 5, bootstrapping once every few iterations would lead to larger parameters and ciphertexts.

\subsection{Encrypted protocol}

Protocol~\ref{prot:dist_ADMM} describes the steps for privately solving the distributed Lasso problem.

We use an optimized diagonal method~\cite{Halevi2014algorithms} for encrypted matrix-vector multiplication. Consider a matrix $\mb S\in\mbR^{n\times n}$ and a vector $\mb p\in \mbR^n$. Denote the diagonals of $\mb S$ by $\mb d_i$, for $i = 0,\ldots,n-1$. Let $n_1 := \lceil \sqrt{n/2} \rceil$ and $n_2 := n/n_1$. The corresponding result $\mb q = \mb S \mb p$ is:
{\medmuskip=1mu
\begin{align}
&\mb q = \sum_{i=0}^{n-1} \mb d_i \odot \rho(\mb p, i) \label{eq:naive_mult}\\
&= \sum_{j=0}^{n_2-1} \rho \left (\sum_{k=0}^{n_1-1} \rho(\mb d_{j \cdot n_1 + k},-j \cdot n_1) \odot \rho(\mb p, k); j \cdot n_1 \right).\label{eq:bsgs_mult}
\end{align}
}\ignorespaces
With~\eqref{eq:bsgs_mult}, we need $n_1 + n_2 = O(\sqrt{n})$ homomorphic rotations, given $\rho(\mb d_{j \cdot n_1 + k},-j \cdot n_1)$, instead of $n$ if we use~\eqref{eq:naive_mult}. In both cases we require $n$ homomorphic multiplications.

For a rectangular matrix, we need extended diagonals but the method is the same. The header of the function that achieves this is $\mr{MultDiag}(\mb S, \mb p)$, and we pass the matrix $\mb S$ as separate ciphertexts encoding diagonals rotated accordingly and the vector $\mb p$ encoded in a ciphertext with trailing zeros. (In line 9 in Protocol~\ref{prot:dist_ADMM}, a masking is performed in order to satisfy the latter requirement.) We implement $\mr{MultDiag}$ such that it returns a ciphertext that encodes the result $\mb q$ with trailing zeros. 
This method can be parallelized. $\mr{MultDiag}$ is performed locally at the servers (lines 3 and 11). 

The header $\mr{ApproxSoftT}(\mb p)$ represents the implementation of the Chebyshev interpolation element-wise for $\mb p$, for a given set of coefficients that specify the degree of the approximation and an interval for which the coefficients are valid. Internally, the input is normalized to the interval $[-1,1]$ (this also masks the junk elements), such that the output is a ciphertext encoding the result with trailing zeros. $\mr{ApproxSoftT}$ is performed locally (line 8 in Protocol~\ref{prot:dist_ADMM}).

$\mr{DBoot}(\mb p)$ is a distributed protocol, where all servers start with the ciphertext of the vector $\mb p$ and all servers obtain a ciphertext that contains the refreshed vector $\mb p$ having a predetermined number of levels (line 6 in Protocol~\ref{prot:dist_ADMM}). It implies one round of communication between all servers, as described in Section~\ref{subsec:FHE}. 

There is also an offline protocol that computes the input~as listed in Protocol~\ref{prot:dist_ADMM}. We mention that the client distributes the rows of $\mb A$ and $\mb b$ to the servers and the shares of the private key. The servers compute $\mb m_i$ by multiplication and $\mb M_i$ by the matrix inversion lemma and a high degree polynomial approximation for the inversion function, and use rotation and masking in order to obtain the diagonal representation needed. The servers collaborate to bootstrap the ciphertexts, in order to start Protocol 1 on the desired level.

\begin{protocol}[Distributed encrypted protocol for~\eqref{eq:Lasso-split-dist} with \emph{equal servers} and \emph{equal data split}]
  \label{prot:dist_ADMM}
  \begin{algorithmic}[1]
\small
 \Require{Public parameters: public key $\mr{pk}$, number of servers $K$, number of maximum iterations $K_{\mr{iter}}$. 
 $C$: secret key $\mr{sk}$. 
 $S_1,\ldots,S_K$: encryption of $\mb M_i = \rho(\mb A_i^\intercal \mb A_i + \rho \mb I)^{-1}$, encryption of $\mb m_i = \frac{1}{\rho}(\mb A_i^\intercal \mb b_i)$, share of the secret key $\mr{sk}_i$, the Chebyshev coefficients for evaluating $S_{\lambda/\rho}(\cdot)$ on a given interval.}
 \Ensure{$C$: $\mb x^{\ast}$.}
 \State $S_{i = 1,\ldots,K}$: set initial values $\mr{E_{v0}}(\mb x^0_i)$, $\mr{E_{v0}}(\mb w^0_i)$, $\mr{E_{v0}}(\mb z^0)$ (the value of $\mb z^k$ is previously agreed upon); 
 \For {$k = 0,\ldots, K_{\mr{iter}}-1$}
 \State $S_{i = 1,\ldots,K}$: $\mr{E_{v0}}(\mb x_i^k) = \mr{MultDiag}(\mb M_i, \mb m_i + \mb z^k - \mb w_i^k)$; 
 \State $S_{i = 1,\ldots,K}$: compute and send to other servers the rotation of the sum $\mr{E_{v0}}(\mb v_i) := \rho(\mb x^{k+1}_i + \mb w^k_i, -(i-1)n)$;
 \State $S_{i = 1,\ldots,K}$: assemble $\mr{E_{v\ast}}(\mb v) := \mr{E_{v\ast}}([\mb v_1 \, \mb v_2 \, \ldots \mb v_K])$ by summing own ciphertext and all received shifted ciphertexts;
 \State $S_{i = 1,\ldots,K}$: perform own part in the distributed bootstrapping and get $\mr{E_{v\ast}}(\mb v^b):=\mr{DBoot}(\mr{E_{v\ast}}(\mb v))$;
 \State $S_{i = 1,\ldots,K}$: extract its refreshed sum of local iterates $\mr{E_{v*}}(\mb v^b_i) = \rho(\mr{E_{v\ast}}(\mb v^b), (i-1)n )$;
 \State $S_{i = 1,\ldots,K}$: rotate and sum $\mr{E_{v\ast}}(\mb v^b)$ to get $\mr{E_{v\ast}}(\sum_{i=1}^K \mb x^{k+1}_i+ \mb w^k_i)$, then compute $\mr{E_{v0}}(\mb z^k) = \mr{ApproxSoftT}(\frac{1}{K}\sum_{i=1}^K \mb x^{k+1}_i + \mb w^k_i, \lambda/\rho)$;
 \State $S_{i = 1,\ldots,K}$: $\mr{E_{v0}}(\mb w^{k+1}_i) = [\mb 1_{S}^\intercal \, \mb 0^\intercal ]^\intercal \odot \mr{E_{v*}}(\mb v^b_i) - \mr{E_{v0}}(\mb z^k)$; 
 \EndFor
 \State $S_1$: compute $\mr{E_{v0}}(\mb x_1^{K_{\mr{iter}}}) = \mr{MultDiag}(\mb M_1, \mb m_1 + \mb z^{K_{\mr{iter}}-1} - \mb w_1^{K_{\mr{iter}}-1})$ and send it to the client $C$; 
 \State $C$: decrypt and extract $\mb x^\ast$.
\end{algorithmic}
\end{protocol}

\begin{proposition}\label{prop:ADMM_distr}
Protocol~\ref{prot:dist_ADMM} achieves client data confidentiality with respect to semi-honest servers, assuming at least one of the servers is honest.
\end{proposition}

\begin{proof}
We use two theoretical results in the proof. First, we need the underlying homomorphic encryption scheme, CKKS, to be semantically secure, in order to ensure that an adversary that does not have access to the private key of the scheme cannot decrypt or even distinguish ciphertexts encrypting different values. This has been proven in~\cite{Cheon2017homomorphic} and assumes that the Decisional Ring Learning with Errors problem is computationally hard~\cite{Lyubashevsky2010ideal}. We select the scheme parameters to ensure that this problem is hard in practice, specifically that it achieves a security level of 128 bits, according to the Learning with Errors estimator of~\cite{Albrecht15}.
Second, we need that the interactive part of the multi-party CKKS scheme, in our case, the distributed bootstrapping protocol, preserves the indistinguishability of the ciphertexts and does not reveal the private key. This has been proven in~\cite{froelicher2021scalable} and assumes that at least one servers is honest and that the masks used in $\mr{DBoot}$ are statistically hiding. Our adversarial model indeed considers that at most $K-1$ servers can be corrupted and we choose the masks to be 80 bits larger than the messages, while ensuring enough levels such that the result does not overflow.

In Protocol~\ref{prot:dist_ADMM}, the data of the client is sent encrypted to the servers. Since the servers cannot all collude in order~to reveal the private key, and no information about the messages underlying is leaked by viewing or computing on the respective ciphertexts, Protocol~\ref{prot:dist_ADMM} achieves client data confidentiality.
\end{proof}

\section{Case study: data predictive control}
\label{sec:reformulation}

Consider that the client wants to control a linear system with unknown model parameters. The goal is to compute a reference-tracking LQR control, based only on precollected input-output data. 
We can reformulate this control problem inspired by the behavioral framework~\cite{Willems2005note,de2019persistency,Coulson2019data,Berberich2020data}.

We need the concept of a \textit{block-Hankel matrix}. For the input signal $\mb u = \left[ \begin{matrix}  \mb u_0^\intercal & \mb u_1^\intercal &\ldots& \mb u_{T-1}^\intercal  \end{matrix} \right]^\intercal \in\mathbb R^{mT}$ and a positive integer $L$, this is given by the following:
\begin{equation*}
	\mb H_L(\mb u) := \left[\begin{matrix}	\mb u_0		& \mb u_1 	& \ldots & 	\mb u_{T-L}	\\
						   			\mb u_1		& \mb u_2 	& \ldots & \mb	u_{T-L+1}	\\ 
						  			\vdots		& 			&\ddots &	\vdots	\\
						   			\mb u_{L-1}	& \mb u_{L}	& \ldots & \mb u_{T-1}\end{matrix} \right].
\end{equation*} 
By definition, the signal $\mb u$ is persistently exciting of order $L$ if $\mb H_L(\mb u)\in\mathbb R^{mL \times (T-L+1)}$ is full row rank.

For input and output data $\mb u^d \in \mathbb R^{mT}$ and $\mb y^d \in \mathbb R^{pT}$, we construct block-Hankel matrices for $M$ samples for the past data and $N$ samples for the future data, where $S:=T-M-N+1$ and $\mb U_p \in\mathbb R^{mM \times S }, \mb U_f \in\mathbb R^{mN \times S }, \mb Y_p \in\mathbb R^{pM \times S }, \mb Y_f \in\mathbb R^{pN \times S }$:
\begin{equation}\label{eq:past+future}
\mb H_{M+N}(\mb u^d) = :\left[\begin{matrix}\mb U_p \\ \mb U_f \end{matrix} \right], \quad \mb H_{M+N}(\mb y^d) = :\left[\begin{matrix}\mb Y_p  \\ \mb Y_f \end{matrix} \right].\\
\end{equation}

Assume we have \textit{data richness}, i.e., the precollected input is persistently exciting~\cite{Willems2005note}. This requires that the precollected input signal has length at least $(m+1)(M+N+n)-1$.

Fix a time $t$ and let $\bar{\mb u}_t=\mb{u}_{t-M:t-1}$ be the batch vector of the last $M$ inputs. The batch vector $\bar{\mb y}_t$ of the last $M$ outputs is defined similarly. If $M\ge n$, the standard LQR problem can be reformulated as
the data predictive control problem~\cite{Coulson2019data} in~\eqref{eq:optimization0}, where the state representation is replaced with the precollected data. According to the behavioral framework, an input-output trajectory of a linear system is in the image of the block-Hankel matrices for the~precollected data, i.e., the constraint in~\eqref{eq:optimization0}, where $\mb g$ is a preimage of~the trajectory. 
$\mb Q$, $\mb R$ are LQR costs and $\mb r$ is the desired~reference. 
The first $m$ elements of $\mb u^{\ast,t}$ are input into the system in a receding horizon fashion and $\mb y^{\ast,t}$ is the predicted output.
{\medmuskip=1mu
\begin{align}\label{eq:optimization0}
\begin{split}
\min_{\mb g,\mb u,\mb y}~&~\frac{1}{2} \sum_{k=t}^{N+t-1} \left( || \mb y_k - \mb r_{k}||_{\mb Q}^2 + || \mb u_k ||_{\mb R}^2 \right) \\
s.t.~&~\left[\mb U_p^\intercal \,\,\, \mb Y_p^\intercal \,\,\, \mb U_f^\intercal \,\,\, \mb Y_f^\intercal \right]^\intercal \cdot \mb g = \Big[\bar{\mb u}_t^\intercal \,\,\, \bar{\mb y}_t^\intercal \,\,\, \mb u^\intercal \,\,\, \mb y^\intercal \Big]^\intercal. 
\end{split}
\end{align}
}

In practice, there will be noise affecting the output measurement, as well as precision errors induced by encryption, which might prevent an exact solution to the equality constraints for the past data in~\eqref{eq:optimization0}. 
Hence, we prefer a relaxation of the equality constraints via an $\ell_2$-least-squares approach with penalty weights $\lambda_y$ and $\lambda_u$.
Then, we rewrite~\eqref{eq:optimization0} as a minimization problem depending only on $\mb g$ by enforcing $\mb u = \mb U_f \mb g$ and $\mb y = \mb Y_f \mb g$. We also batch the objective function, using the same notation $\mb Q,\mb R,\mb r$ for the batched costs and reference. 

Finally, to avoid overfitting due to noisy data, we penalize the magnitude of $\mb g$ through an $\ell_1$-regularization with penalty parameter $\lambda_g$ in~\eqref{eq:optimization_norm1}. The intuition behind this choice comes from the fact that in the noiseless data predictive control formulation, the block-Hankel matrix of the trajectory data has an inherent low-rank structure. Choosing an $\ell_1$-regularization acts like a convex relaxation of imposing a low-rank constraint--see~\cite[Thm. 4.6]{dorfler2021bridging} for more details.
\begin{align}\label{eq:optimization_norm1}
\begin{split}
\min_{\mb g}~& \frac{1}{2} \left( ||\mb Y_f\mb g-\mb r_t||_{\mb Q}^2 + ||\mb U_f\mb g||_{\mb R}^2\right) + \\
&~\lambda_y || \mb Y_p \mb g - \bar{\mb y}_t ||_2^2 + \lambda_u || \mb U_p \mb g - \bar{\mb u}_t ||_2^2 + \lambda_g || \mb g ||_1.
\end{split}
\end{align}

Notice that~\eqref{eq:optimization_norm1} is indeed a Lasso problem:
\begin{align}\label{eq:ddc-Lasso}
\begin{split}
\min_{\mb g}~& \frac{1}{2} ||\mb H\mb g-\mb J \mb f_t||_2^2 + \lambda_g || \mb g ||_1 ,
\end{split}
\end{align}
where we have 
$\mb J := \mr{blkdiag}\left(2\lambda_y \mb I, \, \mb Q, \, 2\lambda_u \mb I, \,\mb R\right)^{1/2}, 
\mb f_t := \Big[ \bar{\mb y}_t^\intercal \,\,\, \mb r_t^\intercal \,\,\, \bar{\mb u}_t^\intercal\,\,\, \mb 0^\intercal \Big]^\intercal, \mb H := \mb J \left[ \mb Y_p^\intercal \,\,\, \mb Y_f^\intercal \,\,\, \mb U_p^\intercal \,\,\, \mb U_f^\intercal \right]^\intercal$. In the case a hybrid regularization $\lambda_g \|\mb g\|_1 + \mu_g \|\mb g\|_2^2$ is preferred, we can use the same formulation~\eqref{eq:ddc-Lasso} and appropriately modify $\mb H$ and $\mb f_t$.

Our goal is to provide a solution that outsources~to~a~cloud service the computation of the optimal solution $\mb g^{\ast,t}$ of~\eqref{eq:ddc-Lasso} and of $\mb u^{\ast,t}= \mb U_f \mb g^{\ast,t}$, while ensuring client data confidentiality for all time steps, as described in Section~\ref{sec:formulation}.

\section{Encrypted data predictive control}
\label{sec:dd_solution}
Following the discussion in Section~\ref{sec:solution}, we write the problem~\eqref{eq:ddc-Lasso} as a distributed problem with split variables:
\begin{align}\label{eq:ddc-Lasso-split-dist}
\begin{split}
\min_{\mb g_1,\ldots, \mb g_K, \mb z}~& \frac{1}{2}  \sum_{i=1}^K||\mb H_i\mb g_i-(\mb J\mb f)_i||_2^2 + \lambda_g || \mb z ||_1  \\
\text{s.t.}~&~ \mb g_i - \mb z = \mb 0, \quad i=1,2,\ldots,K.
\end{split}
\end{align}

In the context of this data predictive problem, when we perform a homogenous split of the data (a split of equal size), the distributed solution converges very slowly~to the global optimal solution. The reason for that is that~the homogeneous sub-problems have a different optimal solution than the global solution. To gain intuition, consider splitting the component matrices $\mb U_p, \mb U_f, \mb Y_p, \mb Y_f$ of $\mb H$ equally between the $K$ servers. This means that each server solves a local optimization problem for the same system that generated the values, but being given fewer samples than necessary to characterize the behaviour of the system, i.e., losing persistency of excitation. The problem remains when allocating a random set of rows of the equal size to the servers. 

To avoid this issue, we prefer to unequally split the~problem. Specifically, we designate Server~1 to have most of the rows and the rest of the servers to hold fewer. Because the local solution of Server~1 is close to the central solution, the empirical convergence to the optimal solution is much faster. However, the more servers we add, the slower the convergence (if we do not weight contributions differently). A valid option is to have only one server do all the computation, i.e., central ADMM, and request help only for the distributed bootstrapping from the rest of the servers (recall that we require multiple servers both for an efficient distributed bootstrapping and for security of the private key). But since the rest of the servers would be idle while the central server performs the computation, we prefer to distribute some of the computations to them as well.

Let the matrix $\mb H_1$ denote the first $(m+p)M + pN$ rows of matrix $\mb H\in \mbb R^{(m+p)(N+M)\times S}$. We split the remaining rows of $\mb H$ into blocks of $mN/(K-1)$ rows, denoted $\mb H_i$ for $i = 2,\ldots, K$. We similarly split $\mb H^\intercal \mb J \in \mbb R^{S\times (m+p)(N+M)}$ into $\bar{\mb H}_1$ and $\bar{\mb H}_2\ldots,\bar{\mb H}_K$ and $\mb f_t\in \mbb R^m$ into $\mb f_{1,t}$ and $\mb f_{2,t},\ldots,\mb f_{K,t}$. 
We prefer to use more of less powerful devices in order to increase the security threshold (by splitting the secret key into more values) and reduce the cost of operating the cloud service. To this end, we shift some of the computations from the less powerful servers to the more powerful Server 1 and remove online communication between the client and the less powerful servers. Protocol~\ref{prot:DD_dist_ADMM} differs from Protocol~\ref{prot:dist_ADMM} in this different allocation of computation, described below. 

First, we split problem~\eqref{eq:ddc-Lasso-split-dist} such that $\mb f_{i,t} = \mb 0$, for $i=2,\ldots,K$, see~\eqref{eq:ddc-admm-dist_hetero}. Second, Servers $2,\ldots,K$ have an easier offline computation, since they have to invert substantially smaller matrices than Server 1, using the matrix inversion lemma. The bootstrapping step is done the same as in Protocol~\ref{prot:dist_ADMM}, after all parties broadcast their local sums.~However, we let only the more powerful Server 1 perform the summation $\sum_{i=1}^K \mb g_i^{k+1} + \mb w_i^k$ and the evaluation of the soft thresholding approximation, and then send the result $\mb z^{k+1}$ to the other less powerful servers (the ciphertext $\mb z^{k+1}$ will have only $l_B+1$ levels so communication is cheap). Then, all servers continue with the computation of $\mb w_i^{k+1}$ and finish the iteration. 
\begin{align}\label{eq:ddc-admm-dist_hetero}
\begin{split}
    \mb g_1^{k+1} &= \left(\mb H_1^\intercal \mb H_1 + \rho \mb I\right)^{-1} \left( \bar{\mb H}_1^\intercal \mb f_1 + \rho (\mb z^k - \mb w_1^k) \right) \\
    \mb g_i^{k+1} &= \rho\left(\mb H_i^\intercal \mb H_i + \rho \mb I\right)^{-1}  (\mb z^k - \mb w_i^k), ~i = 2,\ldots, K \\    
    \mb z^{k+1} &= \frac{1}{K}S_{\lambda_g/\rho} \left(\sum_{i=1}^K\mb g_i^{k+1} + \sum_{i=1}^K\mb w_i^k\right) \\
    \mb w_i^{k+1} &= \mb w_i^k + \mb g_i^{k+1} - \mb z^{k+1}, ~i = 1,\ldots, K.
\end{split}
\end{align}

Moreover, because of the way we split the time-varying vector $\mb f_t$, such that the elements corresponding to Servers $2,\ldots,K$ are 0, there is no need for them to update with the latest values of $\mb u_t$ and $\mb y_t$. This way, only Server~1 needs to have a connection with the client. The ciphertexts communicated to the client are on level 0 (the predicted input $\mb u^{\ast,t}$), while the ciphertexts communicated from the client ($\mb u_t$ and $\mb y_t$ for assembling $\mb f_{1,t}$) are on level $l_B+2$.

Nevertheless, if the servers have different capacity, the more powerful server will likely have to wait on the other servers for the bootstrapping synchronization (requiring completion of the computation for $\mb g_i^{k+1}$).
In the idle time, the more powerful Server 1 can perform multiple local updates, which heuristically helps with convergence in our problem.

\begin{proposition}\label{prop:DD_distr}
Protocol~\ref{prot:DD_dist_ADMM} achieves client data confidentiality with respect to semi-honest servers, assuming at least one of the servers is honest.
\end{proposition}

The proof follows from the proof of Proposition~\ref{prop:ADMM_distr}, regardless of having the servers perform different tasks, since all tasks involve computations only on encrypted data.

\begin{protocol}[Distributed encrypted protocol for~\eqref{eq:ddc-Lasso-split-dist} with \emph{unequal servers} and \emph{unequal data split} for one time step $t$]
  \label{prot:DD_dist_ADMM}
  \begin{algorithmic}[1]
\small
 \Require{Public parameters: public key $\mr{pk}$, parameters of the system and offline trajectory $m,p,N,M,S$, the number of servers $K$, number of maximum iterations $K_{\mr{iter}}$. 
$C$: $(\mb u_\tau, \mb y_\tau)_{\tau=0,\ldots,t}$. 
$S_1$: encryption of $\mb M_1 = \rho(\mb H_1^\intercal \mb H_1 + \rho \mb I)^{-1}$, encryption of $\mb F_1 = \frac{1}{\rho}(\mb M_1 \mb H_1^\intercal \mb J)$, encryption of $\mb U_f$, encryption of $\bar {\mb y}_t, \mb{r}_t, {\mb u}_t$, 
share of the secret key $\mr{sk}_1$, the Chebyshev coefficients for evaluating the soft threshold function for a given interval and bias $\lambda_g/\rho$. 
$S_2,\ldots,S_K$: encryption of $\mb M_i = \rho(\mb H_i^\intercal \mb H_i + \rho \mb I)^{-1}$, 
share of the secret key $\mr{sk}_i$, for $i=2,\ldots,K$.}
 \Ensure{$C$: $\mb u_{t+1}$}
 \State $C$: send to $S_1$ the ciphertexts $\mr{E_{v0}}(\mb u_t)$, $\mr{E_{v0}}(\mb y_t)$, $\mr{E_{v0}}(\mb r_t)$;
 \State $S_1$: assemble the ciphertext $\mr{E_{v0}}(\mb f_{1,t}) =\mr{E_{v0}}( \left[\bar{\mb y}_t^\intercal \, \mb r_t^\intercal \, \bar{\mb u}_t^\intercal\right]^\intercal)$;
 \State $S_{i = 1,\ldots,K}$: set initial values $\mr{E_{v0}}(\mb g^0_i)$, $\mr{E_{v0}}(\mb w^0_i)$, $\mr{E_{v0}}(\mb z^0)$ (the value of $\mb z^k$ is previously agreed upon); 
 \For {$k = 0,\ldots, K_{\mr{iter}}-1$}
 \State $S_1$: compute $\mr{E_{v0}}(\mb g_1^k) = \mr{MultDiag}(\mb F_1, \mb f_{1,t}) + \mr{MultDiag}(\mb M_1, \mb z^k - \mb w_1^k)$; 
 \State $S_{i = 2,\ldots,K}$: compute $\mr{E_{v\ast}}(\mb g_i^k) =  \mr{MultDiag}(\mb M_i, \mb z^k - \mb w_i^k)$;
 \State $S_{i = 1,\ldots,K}$: compute and send to the other servers the rotation of the sum $\mr{E_{v0}}(\mb v_i) := \rho(\mb g^{k+1}_i + \mb w^k_i, -(i-1)S)$;
 \State $S_{i = 1,\ldots,K}$: assemble $\mr{E_{v\ast}}(\mb v) := \mr{E_{v\ast}}([\mb v_1 \, \mb v_2 \, \ldots \mb v_K])$ by summing own ciphertext and all received shifted ciphertexts;
 \State $S_{i = 1,\ldots,K}$: perform own part in the distributed boostratpping and get $\mr{E_{v\ast}}(\mb v^b) := \mr{DBoot}(\mr{E_{v\ast}}(\mb v))$;
 \State $S_{i = 1,\ldots,K}$: extract its refreshed sum of local iterates $\mr{E_{v*}}(\mb v_i) = \rho(\mr{E_{v\ast}}(\mb v), (i-1)S )$;
 \State $S_1$: rotate and sum the refreshed $\mr{E_{v\ast}}(\mb v^b)$ to obtain $\mr{E_{v\ast}}(\sum_{i=1}^K \mb g^{k+1}_i + \mb w^k_i)$, then compute $\mr{E_{v0}}(\mb z^k) = \mr{EvalApproxSoftT}(\frac{1}{K}\sum_{i=1}^K \mb g^{k+1}_i + \mb w^k_i, \lambda_g/\rho)$;
 \State $S_1$: send to all the other servers $\mr{E_{v0}}(\mb z^k)$;
 \State $S_{i = 1,\ldots,K}$: compute $\mr{E_{v0}}(\mb w^{k+1}_i) = [\mb 1_{S}^\intercal \, \mb 0^\intercal ]^\intercal \odot \mr{E_{v*}}(\mb g_i) - \mr{E_{v0}}(\mb z^k)$;
 \EndFor
 \State $S_1$: compute $\mr{E_{v0}}(\mb u^\ast) = \mr{MultDiag}(\mb U_f, \mb g^{K}_1)$ and send the result to the client $C$; \Comment{or directly obtain only the first $m$ components by multiplying by $[\mb I_m \,\, \mb 0]\mb U_f$}
 \State $C$: decrypt $\mb u_{t+1}$, plug it in the system to measure $\mb y_{t+1}$.
\end{algorithmic}
\end{protocol}

\section{Numerical results}
We consider a data-driven temperature control of a 4x4 stable system representing a building with four rooms, with sampling time $300$ seconds, and $M = 4$, $N = 8$, $T = 84$. We add process noise and measurement noise, both zero mean Gaussian with covariance $0.01\mb I$. We choose the cost matrices and regularization terms $\mb Q = 300\mb I, \mb R = \mb I, \lambda_g = 300, \lambda_y = \lambda_u = 3000$.
The data was distributed among 3 servers: Server 1 holds 64 rows and Servers 2 and 3 hold 16 rows each. Convergence for the Lasso problem associated to one time step of the above problem occurred after 20 ADMM iterations, for $\rho = 1200$. Figure~\ref{fig:comparison} reflects the tracking performance of the data predictive control problem with these parameters.

\begin{figure}[tb]
	\centering
	\includegraphics[width=\columnwidth]{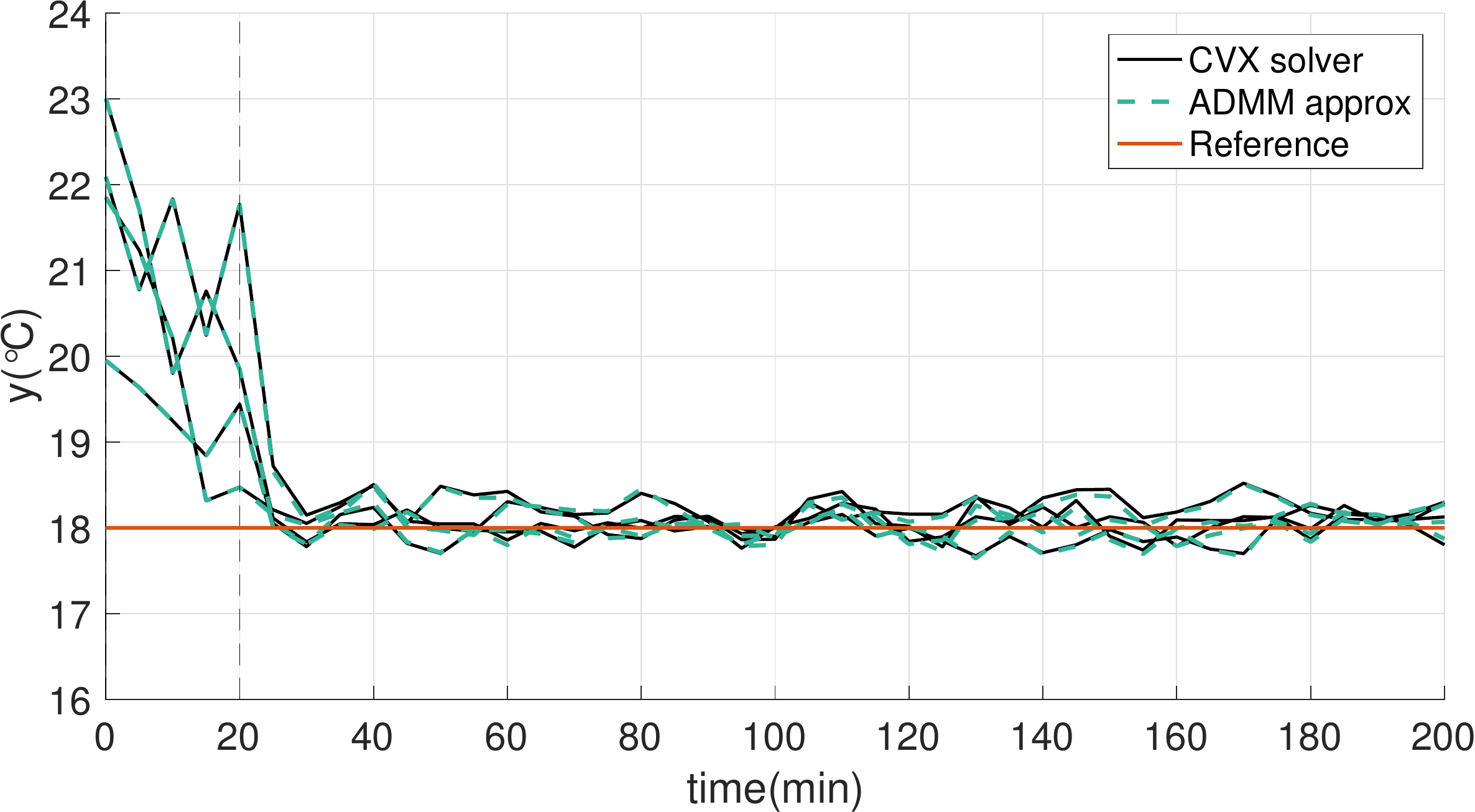}
	\caption{Comparison between the tracking performance of the data predictive controller solving~\eqref{eq:ddc-Lasso} exactly via the CVX solver, and solving~\eqref{eq:ddc-Lasso} via distributed ADMM with 3 servers and approximating the soft thresholding function with a degree-11 polynomial. The vertical dashed line marks the first $M$ time steps, corresponding to the initial offline data. The curves represent the temperature measurements in the four rooms of the system.}
	\label{fig:comparison}
\end{figure}

We evaluate Protocol~\ref{prot:DD_dist_ADMM} on Ubuntu 18.04 on a commodity laptop with 8 GB of RAM and Intel Core i7, implemented using the PALISADE library \cite{PALISADEv10}, using 8 threads. We set the parameters such that we get a security level of 128 bits, i.e., we use a ciphertext modulus of 436 bits and a ring dimension of $2^{14}$. We obtain 6 decimal places precision for the results. The average time for the first iteration is 2.75 seconds and for any iteration afterwards is 1.98 seconds. The time for Server 1 to assemble the vector $\mb f_{1,t}$ from the client and to compute the prediction is 0.61 seconds. The client needs 0.07 seconds to decrypt the control input and to encrypt the new measurement and the input. This gives a total computation time of 38.44 seconds per solving the optimization problem, not taking communication into account. The setup takes 4.5 seconds, and is performed once for all subsequent iterations.

If we artificially add a 150 ms delay of communication (serialization/deserialization and transport) and assume the machines send the messages sequentially to the other machines, then the total computation time increases by 6.45 seconds (450 ms delay per iteration, and 450 ms delay for communication between Server 1 and Client).  

Overall, the maximum amount of memory Server 1 needs to have is 1.22 GB, while Server 2 and 3 need 0.52 GB.

To simulate less powerful devices, we run Servers 2 and 3 on 2 threads instead of 8. The total time necessary for the 20 iterations increases to 49.96 seconds. The majority of the difference comes from the final operation of bootstrapping (which can be done asynchronously, i.e., it occurs after the communication so servers do not have to wait for each other): the total bootstrapping time increases from 0.91 seconds to 1.48 seconds. The rest of the difference comes from the fact that computing $\mb g_2^{k+1}$ and $\mb g_3^{k+1}$ takes 1.06 seconds compared to the 0.83 seconds that Server 1 needs to compute $\mb g_1^{k+1}$. 

\begin{figure}[tb]
	\centering
	\includegraphics[width=\columnwidth]{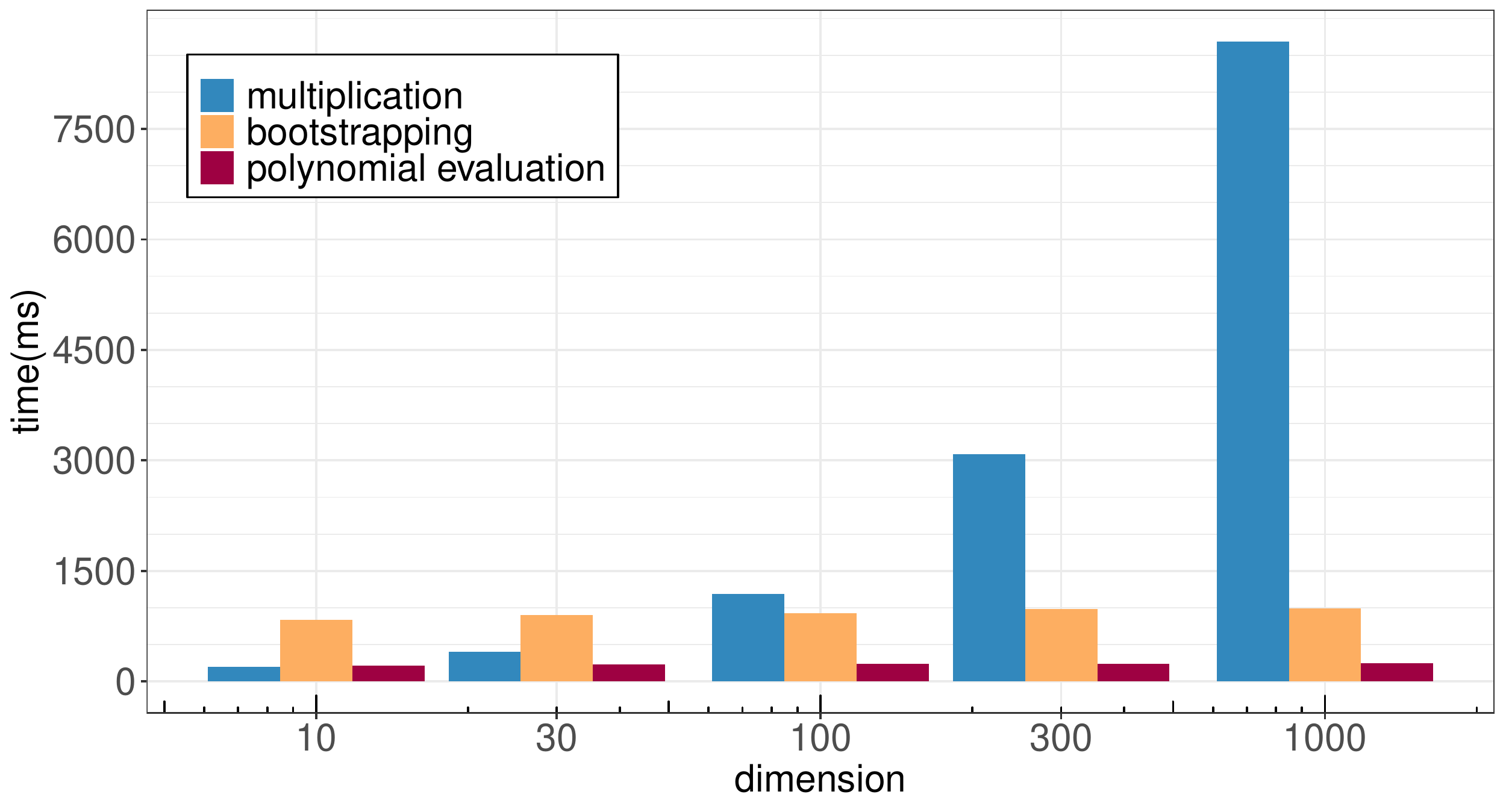}
	\caption{Timing for steps in one online iteration for one server for solving Lasso problems of various dimensions via encrypted distributed ADMM with three servers (Protocol~\ref{prot:dist_ADMM}). The legend shows the operation that takes the most time in the step. }
	\label{fig:ADMM}
\end{figure}

Because the servers only need to synchronize in order to bootstrap (Server 2 and 3 also wait for $\mb z^{k+1}$ from Server 1, but in our instance this does not create idle time since Server 1 is more powerful), Server 1 can either wait for the other servers to finish the computation or can perform two local updates of $\mb g_1$ . Heuristically, since the local solution of Server 1 is closer to the global optimal solution, performing more local iterations helps convergence (in our example, by needing 18 iterations instead of 20). Nevertheless, the computation of the control input is still ready in one sixth of the sampling time.

In Figure~\ref{fig:ADMM}, we show how the time for one ADMM iteration varies with the dimension of the problem, i.e., number of columns of matrix $\mb A$ in~\eqref{eq:Lasso-split}. The scheme parameters are the same as described above. The blue bar shows the time for lines 3 and~4 in Protocol~\ref{prot:dist_ADMM}, effectively consisting of the matrix-vector multiplication. The yellow bar shows the time for lines 5~and 6, representing the preparation for bootstrapping and the bootstrapping itself. Finally, the red bar represents lines 7--9 of Protocol~\ref{prot:dist_ADMM}, consisting mostly of the polynomial evaluation. We want to stress that the bootstrapping and polynomial evaluation are made independent from the dimension of the problem through packing, which represents a great advantage when increasing the dimension. On the other hand, for large dimensions, the encrypted matrix multiplication takes most of the computational and memory effort, and other methods that decrease storage and number of operations at the cost of more levels might be preferable.

\section{Future work}
\label{sec:discussion}
In this work, as well as in most of encrypted protocols for control and optimization, we assumed knowledge~of~some bounds on the variables, which can themselves leak information. An important avenue of research is to design encrypted and differentially private prior experiments~in~order to compute differentially private bounds on parameters such as costs, penalties and number of iterations. We can involve the client in these experiments, by having it control~a small device with one share of the secret key. In this way, the client has to first agree on the computation that is being effectuated to let the experiment continue and decrypt the result. 

We will also perform several optimizations to reduce the memory consumption at the server machines, to make it even more amenable to small devices.

\bibliographystyle{IEEEtran}
\bibliography{IEEEabrv,biblo}
\end{document}